\documentclass[11pt,oneside]{amsart}
\usepackage[latin9]{inputenc}
\usepackage{color}
\usepackage{enumitem}
\usepackage{amsthm}
\usepackage{amstext}
\usepackage{amssymb}
\usepackage{graphicx}
\usepackage[unicode=true,
 bookmarks=true,bookmarksnumbered=false,bookmarksopen=false,
 breaklinks=false,pdfborder={0 0 0},backref=false,colorlinks=true]
 {hyperref}

\makeatletter
\numberwithin{equation}{section}
\numberwithin{figure}{section}
\theoremstyle{plain}
\newtheorem{thm}{\protect\theoremname}[section]
  \theoremstyle{plain}
  \newtheorem{prop}[thm]{\protect\propositionname}
  \theoremstyle{remark}
  \newtheorem{rem}[thm]{\protect\remarkname}
  \theoremstyle{plain}
  \newtheorem{lem}[thm]{\protect\lemmaname}
  \theoremstyle{plain}
  \newtheorem{cor}[thm]{\protect\corollaryname}
  \theoremstyle{definition}
  \newtheorem{example}[thm]{\protect\examplename}
 \newlist{casenv}{enumerate}{4}
 \setlist[casenv]{leftmargin=*,align=left,widest={iiii}}
 \setlist[casenv,1]{label={{\itshape\ \casename} \arabic*.},ref=\arabic*}
 \setlist[casenv,2]{label={{\itshape\ \casename} \roman*.},ref=\roman*}
 \setlist[casenv,3]{label={{\itshape\ \casename\ \alph*.}},ref=\alph*}
 \setlist[casenv,4]{label={{\itshape\ \casename} \arabic*.},ref=\arabic*}

\usepackage{diagrams}

\newcommand*{\ddb}{\frac{\sqrt{-1}}{2\pi}\partial \bar{\partial}}

\newcommand*{\wifi}{\varphi}
\newcommand*{\pp}{\partial}

\date{}

\makeatother

  \providecommand{\corollaryname}{Corollary}
  \providecommand{\examplename}{Example}
  \providecommand{\lemmaname}{Lemma}
  \providecommand{\propositionname}{Proposition}
  \providecommand{\remarkname}{Remark}
 \providecommand{\casename}{Case}
\providecommand{\theoremname}{Theorem}

\begin{document}

\title{The J-flow On Toric Manifolds }

\author{Yi Yao}
\begin{abstract}
We study the J-flow on the toric manifolds, through study the transition
map between the moment maps induced by two K\"{a}hler metrics, which is
a diffeomorphism between polytopes. This is similar to the work of
Fang-Lai, under the assumption of Calabi symmetry, they study the
monotone map between two intervals. We get a partial bound of the
derivatives of transition map.
\end{abstract}
\maketitle

\section{Introduction}

In \cite{moment map}, Donaldson described the various situations
where the diffeomorphism groups act on some spaces of maps between
manifolds, these actions induce the moment maps, then various geometric
flows arise as the gradient flow of the norm square of moment maps,
and the J-flow is one of these. Moreover, in the study of the K-energy,
Chen \cite{lower bound} introduce J-flow as the gradient flow of
the J-functional. Let $X$ be a K\"{a}hler manifold with a K\"{a}hler class
$[\omega]$, the space of K\"{a}hler potentials is
\[
\mathcal{H}=\{\varphi\mid\omega_{\varphi}=\omega+\sqrt{-1}\partial\bar{\partial}\varphi>0\}
\]
Let $\alpha$ be a K\"{a}hler metric, the J-functional is defined on $\mathcal{H}$
by
\[
\mathcal{J}_{\alpha,\omega}(\varphi)=\int_{0}^{1}\int_{X}\dot{\varphi_{t}}(\alpha\wedge\omega_{\varphi_{t}}^{n-1}-c\omega_{\varphi_{t}}^{n})\frac{dt}{(n-1)!}
\]
where $\{\varphi_{t}\}_{0\leq t\leq1}$ be any smooth path in $\mathcal{H}$
from $0$ to $\varphi$, and $c=\frac{\int\omega^{n-1}\wedge\alpha}{\int\omega^{n}}$,
so $\mathcal{J}_{\alpha,\omega}(\varphi)=\mathcal{J}_{\alpha,\omega}(\varphi+a)$.
The critical point $\varphi$ of $\mathcal{J}_{\alpha,\omega}$ should
satisfy the Donaldson's equation
\begin{equation}
c\omega_{\varphi}^{n}=\alpha\wedge\omega_{\varphi}^{n-1}\label{eq:donaldson equation}
\end{equation}
And the J-flow is
\begin{equation}
\begin{cases}
\frac{\partial\varphi}{\partial t}=nc-\frac{n\omega_{\varphi}^{n-1}\wedge\alpha}{\omega_{\varphi}^{n}}\\
\varphi|_{t=0}=\varphi_{0}
\end{cases}\label{eq:J-flow}
\end{equation}

Chen \cite{parabolic flow} showed the long time existence of J-flow
and the convergence to solution of (\ref{eq:donaldson equation})
when $\alpha$ have non-negative bisectional curvature. Then in the
work of Song-Weinkove \cite{SW}, by a more delicate estimate based
on the previous work of Weinkove \cite{W1,W2}, a necessary and sufficient
condition for convergence is found,

J-flow (\ref{eq:J-flow}) converges to the solution of (\ref{eq:donaldson equation})
if and only if there exists a metric $\omega'\in[\omega]$ such that
the $(n-1,n-1)$ type form

\begin{equation}
nc\omega'^{n-1}-(n-1)\omega'^{n-2}\wedge\alpha>0\label{eq:song weinkove condition}
\end{equation}
This condition almost assume that there exists a subsolution of (\ref{eq:donaldson equation}).
However it is hard to check for concrete example. In particular, it
is hard to see from this condition that whether the convergence depends
on the choice of $\alpha$ in its class $[\alpha]$, in other words,
if the solvability of (\ref{eq:donaldson equation}) only depends
on the class $[\omega]$ and $[\alpha]$.

In \cite{STABILITY}, Lejmi and Sz\'{e}kelyhidi study the solvability
of (\ref{eq:donaldson equation}) from the view of geometric stability,
as the problem of the existence of cscK metrics, that is conjectured
be equivalent to the K-stability of manifold. Let $L$ be a line bundle
on $X$, $[\omega]=c_{1}(L)$, for a test-configuration $\chi$ for
$(X,L)$, they define an invariant $F_{\alpha}(\chi)$ which is similar
to the Donaldson-Futaki invariant when we study the K\"{a}hler-Einstein
metrics with conical singularity. It is proved that if (\ref{eq:donaldson equation})
have a solution then $F_{\alpha}(\chi)>0$ for any test-configuration
$\chi$ with positive norm, this is corresponding to the result in
\cite{filtration} which is for the cscK metrics. In particular, when
$\chi$ is coming from the deformation to the normal cone of a subvariety
(see \cite{ross an obsruction}), the corresponding condition $F_{\alpha}(\chi)>0$
is

For all $p$-dimensional subvariety $V$ of $X$, where $p=1,2,\cdots,n-1$,
we have
\begin{equation}
nc\int_{V}\frac{\omega^{p}}{p!}>\int_{V}\frac{\omega^{p-1}}{(p-1)!}\wedge\alpha\label{eq:gabor 's condition}
\end{equation}
Obviously these conditions only depend on the classes. They also conjecture
that (\ref{eq:gabor 's condition}) would be sufficient condition
for (\ref{eq:donaldson equation}) have solution. Moreover, (\ref{eq:gabor 's condition})
can be derived directly, suppose (\ref{eq:donaldson equation}) have
solution $\varphi$, for smooth point $x\in V$, choose coordinate
$z^{i}$ such that $V=\{z^{p+1}=\cdots=z^{n}=0\}$ near $x$, and
$\omega_{\varphi,i\bar{j}}=\delta_{i\bar{j}}$ at $x$, since
\[
p\omega_{\varphi,V}^{p-1}\wedge\alpha_{V}=tr_{\omega_{\varphi,V}}\alpha_{V}\ \omega_{\varphi,V}^{p}
\]
where $\alpha_{V}$ is the restriction of $\alpha$ on $V$. The trace
$tr_{\omega_{\varphi,V}}\alpha_{V}=\sum_{i\leq p}\alpha_{i\bar{i}}<\sum\alpha_{i\bar{i}}=nc$,
so $nc\omega_{\varphi,V}^{p}-p\omega_{\varphi,V}^{p-1}\wedge\alpha_{V}>0$
at $x$, then integrate it over $V_{reg}$ is (\ref{eq:gabor 's condition}).
In the same way, we see $nc\omega_{\varphi}-\alpha>0$, so on the
class level $nc[\omega]-[\alpha]>0$. When $n=2$, by (\ref{eq:song weinkove condition})
this is a necessary and sufficient condition for solving (\ref{eq:donaldson equation}),
but when $n>2$ it is not sufficient, see counter-example in \cite{STABILITY}.

When $n=2$, Donaldson \cite{moment map} noted that the above condition
$nc[\omega]-[\alpha]>0$ is satisfied for all K\"{a}hler classes if there
not exist curves with negative self-intersection, and conjectured
that if this condition is violated, the flow (\ref{eq:J-flow}) will
blow up over these curves.

In \cite{SW}, they confirm the above conjecture in a partial sense.
More recently, in \cite{FLSW,newest}, they consider the situations
where $nc[\omega]-[\alpha]\geq0$ or $\alpha$ degenerate along a
divisor, it is proved that the flow will converge outside a union
set of curves. The argument heavily depends on the fact that when
$n=2$ (\ref{eq:donaldson equation}) can be transformed to the Monge-AmpšŠre
equation, and the latter has continuous solution in the degenerate
case, due to the work of Eyssidieux-Guedj-Zeriahi.

For the concrete example, Fang and Lai \cite{ansatz} study the long
time behavior of J-flow on the projective bundles $X_{m,n}=\mathbb{P}(\mathcal{O}_{\mathbb{P}^{n}}\oplus\mathcal{O}_{\mathbb{P}^{n}}(-1)^{\oplus(m+1)})$,
$X_{0,n}$ is the $\mathbb{P}^{n+1}$ blow-up one point. Under the
Calabi symmetry assumption, the flow can be described by a time-dependent
monotone map between two intervals, through solve the static equation
(ODE), they see the flow always converges to a K\"{a}hler current and
on its smooth region satisfies (\ref{eq:donaldson equation}).

In this paper, we study J-flow on the toric manifolds and assume the
metrics are invariant under torus action. We expect to find more verifiable
conditions which can ensure the flow converges, maybe (\ref{eq:gabor 's condition})
for the invariant subvarieties or some combinatorial conditions for
polytopes. If the flow dose not converge, we also want to understand
its asymptotic behavior.

After the symmetry reduction, (\ref{eq:J-flow}) can be defined on
$\mathbb{R}^{n}$ by

\[
\frac{\partial\phi_{t}}{\partial t}=nc-\sum_{i,j}f_{ij}\phi_{t}^{ij}
\]
where $\phi_{t}$ and $f$ are potentials for $\omega_{\varphi_{t}}$
and $\alpha$ respectively, they conform the asymptotic behavior at
infinity assigned by the polytope $\mathcal{P}$ and $\mathcal{Q}$.
Through the Legendre transform of $\phi_{t}$, this nonlinear equation
be transformed to a quasilinear one which is defined on $\mathcal{P}$,
\begin{equation}
\frac{\partial u_{t}}{\partial t}=\sum_{i,j}f_{ij}(\nabla u_{t})u_{ij}-nc\label{eq:intro quasi-linear}
\end{equation}
$u_{t}$ is the Legendre transform of $\phi_{t}$, satisfies Guillemin's
boundary condition, so $\nabla u_{t}$ blow up near the boundary,
since $[f_{ij}]$ degenerate at infinity, so (\ref{eq:intro quasi-linear})
is degenerate on the boundary, the RHS is even not defined on the
boundary. Moreover, when the flow dose not converge, by the example
\ref{in-,-they  example Fang and Lai}, we see $\nabla u_{t}$ may
blow up in a whole domain located in $\mathcal{P}$ as $t\rightarrow\infty$.

We turn to study the transition map $U_{t}=\nabla f\circ\nabla u_{t}$
between the moment maps induced by $\omega_{\varphi_{t}}$ and $\alpha$,
it is a diffeomorphism between polytopes and map the face to face.
The price is $U_{t}$ satisfies a degenerate parabolic system (\ref{eq:Evolute equation parabolic}).
The static map satisfies
\[
trDU\equiv nc
\]
In the paper we just get a partial bound on $DU_{t}$, we conjecture
that $DU_{t}$ is bounded uniformly w.r.t. time, and $U_{t}$ will
converge to a limit map $U_{\infty}$ even if the origin J-flow does
not converge, but in this case $U_{\infty}$ must degenerate on some
domain in the sense $\det DU_{\infty}=0$, since if $\det DU>\delta$
uniformly imply the flow converges. This degeneracy may violate the
condition (\ref{eq:gabor's condition toric}). We also want to know
if J-flow minimizes the functional, namely if
\[
E_{\alpha}(\omega_{\varphi_{t}})\rightarrow\inf_{\omega\in[\omega]}E_{\alpha}(\omega)
\]
 where
\[
E_{\alpha}(\omega)=\frac{1}{2}\int_{X}(tr_{\omega}\alpha)^{2}\frac{\omega^{n}}{n!}=\frac{1}{2}\int_{\mathcal{P}}(trDU)^{2}\ dy
\]
If it does, $trDU_{\infty}$ may correspond to the worst test-configuration
which be discussed in \cite{STABILITY}.

\section{Toric Manifolds and Potentials}

We review the K\"{a}hler structure on toric manifolds in detail, since
we need the coordinate charts which include the invariant divisors,
the logarithmic coordinate defined on the dense open set push these
divisors to infinity. These coordinates also been introduced in \cite{toric surfaces},
here we give a basis-free formulation. More details of toric variety
see \cite{fulton}.

Let $(\mathbb{C},\times)$, $(\mathbb{R}_{\geq},\times)$ be semi-groups,
$\mathbb{R}_{\geq}$ is the set of non-negative real number. $N$
is a lattice with rank $n$, its dual lattice is $M$, given a Delzant
polytope $\mathcal{P}$ in $M_{R}=M\otimes_{\mathbb{Z}}\mathbb{R}$,
and suppose $\{u_{i}\}\subset N$ be the prime inward normal vectors
of the facets of $\mathcal{P}$, then $\mathcal{P}$ is
\begin{equation}
\mathcal{P}=\{y\in M_{R}\mid d_{i}(y)=\langle u_{i},y\rangle+b_{i}\geq0,\ \textrm{for\ all}\ i\}\label{eq:polytope defini}
\end{equation}
it induces a fan $\Sigma$ (a collection of cones) in $N_{R}=N\otimes_{\mathbb{Z}}\mathbb{R}$.
For a vertex $q$ of $\mathcal{P}$, it corresponds a n-dimensional
cone $\sigma_{q}$ in $\Sigma$,

\[
\sigma_{q}=\{u\in N_{R}\mid\langle u,q\rangle=\min_{\mathcal{P}}\langle u,\cdot\rangle\}
\]
its dual cone $\sigma_{q}^{\vee}=\{y\in M_{R}\mid\langle\sigma_{q},y\rangle\geq0\}$
is generated by $\mathcal{P}-q$, and the semi-group $\sigma_{q}^{\vee}\cap M$
is finitely generated, so we can construct a finitely generated algebra
$\mathbb{C}[\sigma_{q}^{\vee}\cap M]=\{\sum_{v}a_{v}\chi^{v}\mid v\in\sigma_{q}^{\vee}\cap M,\ a_{v}\in\mathbb{C}\}$
with multiplication $\chi^{v}\cdot\chi^{v'}=\chi^{v+v'}$, it defines
an affine open set $U_{q}$
\[
U_{q}=\textrm{Spm}\ \mathbb{C}[\sigma_{q}^{\vee}\cap M]=\{\varphi\mid\varphi:\sigma_{q}^{\vee}\cap M\rightarrow\mathbb{C}\}
\]
where $\varphi$ is a homomorphism between semi-groups. Let $e_{1}^{q},\cdots,e_{n}^{q}\in M$
be the prime vectors rooted at $q$ and along the edges of $\mathcal{P}$,
it is a basis of $M$ due to Delzant's conditions, and generates $\sigma_{q}^{\vee}\cap M$.
Assume $\varphi(e_{i}^{q})=z_{i}^{q}$, then $U_{q}\cong\mathbb{C}^{n}=\{(z_{1}^{q},\cdots,z_{n}^{q})\}$.
$U_{q}$ include the dense open subset $U_{0}=\textrm{Spm}\ \mathbb{C}[M]=\{\varphi:M\rightarrow\mathbb{C}^{*}\}\cong N\otimes_{\mathbb{Z}}\mathbb{C}^{*}\cong(\mathbb{C}^{*})^{n}$.
$\{U_{q}\}$ glues with each other to form a toric manifold $X_{\Sigma}$
with a torus $\textrm{Spm}\ \mathbb{C}[M]$ action.

\begin{figure}
\begin{diagram}
&&&&U_0  &\rInto & U_q &  \rInto & X_{\Sigma}\\
&&&&\dTo_{\pi_0}  &  &   \dTo_{\pi_q}   &    &  \dTo_{\pi}  \\
M_R & \lTo^{d\phi} & N_R &\lTo^{\cong}& U_0^{\geq}  &\rInto & U_q^{\geq}  & \rInto & X_{\Sigma}^{\geq}
\end{diagram}

\caption{}

\end{figure}

The subset $U_{q}^{\geq}=\{\varphi:\sigma_{q}^{\vee}\cap M\rightarrow\mathbb{R}^{\geq}\}\cong\mathbb{R}_{\geq}^{n}$
is called the non-negative part of $U_{q}$, $\pi_{q}:U_{q}\rightarrow U_{q}^{\geq}$
is defined by $\varphi\mapsto|\varphi|^{2}$. In the same way we have
$\pi_{0}:U_{0}\rightarrow U_{0}^{\geq}=\{\varphi:M\rightarrow\mathbb{R}^{+}\}\cong N\otimes_{\mathbb{Z}}\mathbb{R}^{+}$.
Since $(\mathbb{R}^{+},\times)\cong(\mathbb{R},+)$ by $a\mapsto\log a$,
we identify $N\otimes_{\mathbb{Z}}\mathbb{R}^{+}$ with $N\otimes_{\mathbb{Z}}\mathbb{R}=N_{R}$.
$\{U_{q}^{\geq}\}$ glues with each other to form a closed subset
$X_{\Sigma}^{\geq}\subset X_{\Sigma}$, and $\{\pi_{q}\}$ glues to
a continuous map $\pi:X_{\Sigma}\rightarrow X_{\Sigma}^{\geq}$.

Let $u$ be a symplectic potential, that is a convex function defined
on $\bar{\mathcal{P}}$ and satisfies
\begin{itemize}
\item restrict $u$ in the interior of a face, it is smooth and strictly
convex.
\item Guillemin boundary condition, $u=\sum d_{i}(y)\log d_{i}(y)+v$, $v\in\mathcal{C^{\infty}}(\bar{\mathcal{P}})$.
\end{itemize}
Then $u$ induces an invariant K\"{a}hler metric $\omega_{u}$ in following
way, let $\phi$ be the Legendre transform of $u$, which is defined
on $N_{R}$, for $x\in N_{R}$,
\[
\phi(x)=\langle x,y\rangle-u(y),\ x=du(y),\ y=d\phi(x)
\]
Take a vertex $q$, let $\phi_{q}=\phi-\langle\cdot,q\rangle$, $\phi_{q}$
can be defined on $U_{0}^{\geq}$ through $U_{0}^{\geq}\cong N\otimes_{\mathbb{Z}}\mathbb{R}^{+}\cong N_{R}$.
Since $U_{0}^{\geq}\subset U_{q}^{\geq}$, $\phi_{q}$ can be extended
smoothly on $U_{q}^{\geq}$ due to the Guillemin's boundary condition,
denote the extension is $\bar{\phi}_{q}$. Let $\Phi_{q}=\bar{\phi}_{q}\circ\pi_{q}$,
then $\omega_{q}=\ddb\Phi_{q}>0$ on $U_{q}$ due to the convexity
of $u$. Since $\phi_{q}-\phi_{q'}=\langle\cdot,q'-q\rangle$, let
$\chi_{q}$ be the pullback of $\langle\cdot,q\rangle$ by the composed
map $U_{0}\rightarrow U_{0}^{\geq}\rightarrow N_{R}$, so on $U_{0}$,
we have $\Phi_{q}-\Phi_{q'}=\chi_{q'-q}$. Take a basis of $M$, $\chi_{q}$
is $(z_{1},\cdots,z_{n})\longmapsto q_{1}\log|z_{1}|^{2}+\cdots+q_{n}\log|z_{n}|^{2}$,
so $\partial\bar{\partial}\chi_{q}=0$. Hence $\omega_{q}=\omega_{q'}$
on $U_{0}$, so is on $U_{q}\cap U_{q'}$ since $U_{0}$ is dense,
so $\{\omega_{q}\}$ defines $\omega_{u}$. In another way, let $\Phi$
be the pullback of $\phi$ by $U_{0}\rightarrow U_{0}^{\geq}\rightarrow N_{R}$,
then $\omega_{u}=\ddb\Phi$ on $U_{0}$, and the extension of $\Phi-\chi_{q}$
to $U_{q}$ is $\Phi_{q}$.

The map $U_{0}^{\geq}\rightarrow N_{R}\stackrel{d\phi}{\rightarrow}M_{R}$
can be extended to $X_{\Sigma}^{\geq}$ as a homeomorphism onto $\bar{\mathcal{P}}$,
compose it with $\pi$ is the moment map $\mu$ of $(X_{\Sigma},\omega_{u})$
with respect to the torus action.

The K\"{a}hler class $[\omega_{u}]=\sum b_{i}[D_{i}]$, $D_{i}$ is the
invariant divisor corresponding to $u_{i}$.

Now we take a basis of $M$ and write down the above things explicitly,
let $\{e_{i}^{q}\}$ be the basis mentioned above, and the dual basis
of $N$ is $\{v_{q}^{i}\}$, it is the prime generators of the cone
$\sigma_{q}$.

Then $U_{q}$ have coordinates $z_{i}^{q}$, $U_{q}^{\geq}\cong\mathbb{R}_{\geq}^{n}$
have coordinates $a_{i}^{q}\geq0$, and $U_{0}^{\geq}=\{(a_{i}^{q})\mid a_{i}^{q}>0\}$,
$N_{R}$ have coordinates $x_{i}^{q}$, $M_{R}$ have coordinates
$y_{q}^{i}$. The map $\pi_{q}:U_{q}\rightarrow U_{q}^{\geq}$ is
$(z_{i}^{q})\mapsto(|z_{i}^{q}|^{2})$, and the identification $U_{0}^{\geq}\rightarrow N_{R}$
is $(a_{i}^{q})\mapsto(\log a_{i}^{q})$.

In the following we omit the index $q$ of variables for simplicity,
$x_{i}$, $a_{i}$ always means $x_{i}^{q}$, $a_{i}^{q}$, etc.

Let
\[
u_{q}(y^{1},\cdots,y^{n})=u(q+y^{1}e_{1}^{q}+\cdots+y^{n}e_{n}^{q}),\ y^{i}\geq0
\]
Then the Legendre transform of $u_{q}$ is $\phi_{q}(\sum x_{i}v_{q}^{i}),\ x_{i}=\partial u_{q}/\partial y^{i}$,
denote it as $\phi_{q}(x_{i})\triangleq\phi_{q}(x_{1},\cdots,x_{n})$
for short. Since the map $U_{0}\rightarrow U_{0}^{\geq}\rightarrow N_{R}$
is $(z_{i})\mapsto\sum\log|z_{i}|^{2}\ v_{q}^{i}$, so $\Phi_{q}(z_{i})=\phi_{q}(\log|z_{i}|^{2})$,
it can be extended smoothly on $\mathbb{C}^{n}$ as a K\"{a}hler potential,
and $\bar{\phi}_{q}(a_{i})=\phi_{q}(\log a_{i})$ can be extended
smoothly on $\mathbb{R}_{\geq}^{n}$. We have three functions $\Phi_{q}(z_{i})$,
$(z_{i})\in\mathbb{C}^{n}$, $\bar{\phi}_{q}(a_{i})$, $(a_{i})\in\mathbb{R}_{\geq}^{n}$,
$\phi_{q}(x_{i})$, $(x_{i})\in\mathbb{R}^{n}$ and satisfy
\[
\Phi_{q}(z_{i})=\bar{\phi}_{q}(a_{i})=\phi_{q}(x_{i}),\ a_{i}=|z_{i}|^{2},\ x_{i}=\log a_{i}
\]

For the convexity of these functions, on $\mathbb{C}^{n}$
\begin{equation}
\left[\frac{\pp^{2}\Phi_{q}}{\pp z_{i}\pp\bar{z}_{j}}\right]=\left[\delta_{ij}\frac{\pp\bar{\phi}_{q}}{\pp a_{i}}+\bar{z_{i}}z_{j}\frac{\pp^{2}\bar{\phi_{q}}}{\pp a_{i}\pp a_{j}}\right]>0\label{eq:Phi_ij}
\end{equation}
On $(\mathbb{C}^{*})^{n}$

\begin{equation}
\left[\frac{\pp^{2}\Phi_{q}}{\pp z_{i}\pp\bar{z}_{j}}\right]=\left[\frac{1}{z_{i}\bar{z_{j}}}\frac{\pp^{2}\phi_{q}}{\pp x_{i}\pp x_{j}}\right]>0\label{eq:phi_ij on C-star}
\end{equation}
$\phi_{q}$ is strictly convex on $\mathbb{R}^{n}$. For $\bar{\phi_{q}}$,
note that when $z_{i}$, $z_{j}\neq0$
\begin{equation}
\delta_{ij}\frac{\pp\bar{\phi}_{q}}{\pp a_{i}}+\bar{z_{i}}z_{j}\frac{\pp^{2}\bar{\phi_{q}}}{\pp a_{i}\pp a_{j}}=\left(\delta_{ij}a_{i}\frac{\pp\bar{\phi}_{q}}{\pp a_{i}}+a_{i}a_{j}\frac{\pp^{2}\bar{\phi_{q}}}{\pp a_{i}\pp a_{j}}\right)\frac{1}{z_{i}\bar{z_{j}}}\label{eq:=005B9E=005316}
\end{equation}
we can see the RHS of (\ref{eq:Phi_ij}) is positive definite if and
only if $\bar{\phi_{q}}$ satisfies
\begin{itemize}
\item $\frac{\pp\bar{\phi_{q}}}{\pp a_{i}}>0$ when $a_{i}=0$
\item for any $\Lambda\subseteq\{1,\cdots,n\}$, it can be empty, on coordinate
plane $\{a_{i}>0,i\in\Lambda,a_{i}=0,i\notin\Lambda\}$,
\[
\left[\delta_{ij}a_{i}\frac{\pp\bar{\phi}_{q}}{\pp a_{i}}+a_{i}a_{j}\frac{\pp^{2}\bar{\phi_{q}}}{\pp a_{i}\pp a_{j}}\right]_{i,j\in\Lambda}>0
\]

\end{itemize}
$\bar{\phi_{q}}$ and $u_{q}$ transform to each other in a way similar
to the Legendre transform. For example, on coordinate plane $\{y\mid y^{n}=0,\ y^{i}\neq0,\ i<n\}$
corresponding to the facet of polytope,
\[
u_{q}(y^{i},0)=\sum_{i<n}y^{i}\log a_{i}-\bar{\phi_{q}}(a_{i},0),\ y^{i}=a_{i}\frac{\partial\bar{\phi}_{q}}{\partial a_{i}}(a_{k},0),\ \textrm{for}\ i<n
\]
So $\bar{\phi_{q}}|_{a_{n}=0}$ is determinate by $u_{q}|_{y^{n}=0}$.

The moment map%
\footnote{We call it as moment map just for convenient.%
} $X_{\Sigma}^{\geq}\rightarrow\bar{\mathcal{P}}$, restrict on $U_{q}^{\geq}\cong\mathbb{R}_{\geq}^{n}$
is

\begin{equation}
(a_{i})\mapsto q+a_{1}\frac{\pp\bar{\phi}_{q}}{\pp a_{1}}e_{1}^{q}+\cdots+a_{n}\frac{\pp\bar{\phi}_{q}}{\pp a_{n}}e_{n}^{q}\label{eq:moment map}
\end{equation}
it is smooth since $\bar{\phi}_{q}$ is smooth on $\mathbb{R}_{\geq}^{n}$.
The inverse map is
\begin{equation}
q+y^{1}e_{1}^{q}+\cdots+y^{n}e_{n}^{q}\mapsto\left(\exp\frac{\pp u_{q}}{\pp y^{i}}\right)\label{eq:inverse of moment map}
\end{equation}
Since $u$ satisfies the Guillemin's boundary condition, $u_{q}=\sum y^{i}\log y^{i}+v(y)$,
$v$ is smooth to the boundary, so $\exp\frac{\pp u_{q}}{\pp y^{i}}=y^{i}\exp(1+\frac{\pp v}{\pp y^{i}})$
is smooth. Hence $X_{\Sigma}^{\geq}\rightarrow\bar{\mathcal{P}}$
is actually a diffeomorphism.

On $U_{0}^{\geq}$, compose (\ref{eq:moment map}) and its inverse
(\ref{eq:inverse of moment map}), we have

\[
\frac{\pp u_{q}}{\pp y^{i}}\left(a_{1}\frac{\pp\bar{\phi}_{q}}{\pp a_{1}},\cdots,a_{n}\frac{\pp\bar{\phi}_{q}}{\pp a_{n}}\right)=\log a_{i}
\]
Take derivative w.r.t. $a_{j}$,
\[
\sum_{k}\frac{\pp^{2}u_{q}}{\pp y^{i}\pp y^{k}}\left(\delta_{kj}a_{j}\frac{\pp\bar{\phi}_{q}}{\pp a_{j}}+a_{k}a_{j}\frac{\pp^{2}\bar{\phi_{q}}}{\pp a_{k}\pp a_{j}}\right)=\delta_{ij}
\]
so on $U_{0}^{\geq}$,
\begin{equation}
\left[\delta_{ij}a_{i}\frac{\pp\bar{\phi}_{q}}{\pp a_{i}}+a_{i}a_{j}\frac{\pp^{2}\bar{\phi_{q}}}{\pp a_{i}\pp a_{j}}\right]^{-1}(a_{k})=\left[\frac{\pp^{2}u_{q}}{\pp y^{i}\pp y^{j}}\right](a_{k}\frac{\pp\bar{\phi}_{q}}{\pp a_{k}})\label{eq:inverse of phi_ij}
\end{equation}
with (\ref{eq:Phi_ij}) and (\ref{eq:=005B9E=005316}), $a_{i}=|z_{i}|^{2}$,
we have
\begin{equation}
\left[\frac{\pp^{2}\Phi_{q}}{\pp z_{i}\pp\bar{z}_{j}}\right]^{-1}(z_{i})=\left[\bar{z}_{i}z_{j}\frac{\pp^{2}u_{q}}{\pp y^{i}\pp y^{j}}\right](a_{i}\frac{\pp\bar{\phi}_{q}}{\pp a_{i}})\label{eq:inverse of Big phi_ij}
\end{equation}
On the coordinate plane we have similar formula, for example on $\{z_{n}=z_{n-1}=0,\ z_{\alpha}\neq0\}$,

\begin{equation}
\left[\frac{\pp^{2}\Phi_{q}}{\pp z_{i}\pp\bar{z}_{j}}\right]^{-1}=\left[\begin{array}{ccc}
\left(\bar{z}_{i}z_{j}\frac{\pp^{2}u_{q}}{\pp y^{i}\pp y^{j}}\right)_{i,j<n-1} & 0 & 0\\
0 & (\frac{\pp\bar{\phi}_{q}}{\pp a_{n-1}})^{-1} & 0\\
0 & 0 & (\frac{\pp\bar{\phi}_{q}}{\pp a_{n}})^{-1}
\end{array}\right]\label{eq:inverse zn=00003Dzn-1=00003D0}
\end{equation}

We know $[u_{q,ij}]$ is singular on the boundary of polytope, however
its inverse can be extended smoothly to the boundary.
\begin{prop}
\label{g^ij IS smooth over polytope}The inverse of $[u_{q,ij}]$
can be extended smoothly on \textup{$\bar{\mathcal{P}}$.}\end{prop}
\begin{proof}
From (\ref{eq:inverse of phi_ij}) and (\ref{eq:inverse of moment map})
we know the inverse
\begin{equation}
[u_{q}^{ij}](y^{k})=\left[\delta_{ij}a_{i}\frac{\pp\bar{\phi}_{q}}{\pp a_{i}}+a_{i}a_{j}\frac{\pp^{2}\bar{\phi_{q}}}{\pp a_{i}\pp a_{j}}\right](\exp\frac{\pp u_{q}}{\pp y^{k}})\label{eq: g^ij formular}
\end{equation}
Obviously it is can be extended to the boundary, moreover $[u_{q}^{ij}]$
is semi-positive definite on the face, and positive definite in the
tangent space of face.
\end{proof}

\section{Transition Between Moment Maps}

Let $\mathcal{P}$ and $\mathcal{Q}$ be defined by (\ref{eq:polytope defini})
with different collection of $b_{i}$, they have similar shape, for
every face of $\mathcal{P}$ there is a corresponding face of $\mathcal{Q}$
parallel to it. Then induce same fan $\Sigma$, so same toric manifold
$X_{\Sigma}$. Let $[\omega]$ and $[\alpha]$ be the corresponding
K\"{a}hler class. The invariant K\"{a}hler metric $\omega$ induces a moment
map $\mu_{\omega}:X_{\Sigma}\rightarrow\bar{\mathcal{P}}$, and $\alpha$
induces another one $\mu_{\alpha}:X_{\Sigma}\rightarrow\bar{\mathcal{Q}}$.
It is well known that $\mu_{\omega}$ establishes a one-to-one correspondence
between the real torus orbits and the points of $\bar{\mathcal{P}}$,
so there is a transition map $U:\bar{\mathcal{P}}\rightarrow\bar{\mathcal{Q}}$
such that $U\circ\mu_{\omega}=\mu_{\alpha}$, it is a homeomorphism.
We next check it is actually a diffeomorphism.

Let $u(y)$ and $g(y)$ be the symplectic potentials of $\omega$
and $\alpha$, they are defined on $\mathcal{\bar{P}}$ and $\mathcal{\bar{Q}}$
respectively. Take a vertex $q$ of $\mathcal{P}$, the corresponding
vertex of $\mathcal{Q}$ is $q'$. It corresponds an affine open set
$U_{q}=U_{q'}$, we have coordinates on it. Without loss of generality%
\footnote{We can translate the polytope such that $q=0$.%
}, we assume $q=q'=0$, so $\phi=\phi_{q}$, $u=u_{q}$. In the following,
we omit the index $q$, $q'$ of potentials, and denote the potentials
$\Phi$, $\bar{\phi}$, $\phi$ for $\omega$, $F$, $\bar{f}$, $f$
for $\alpha$, such that $\omega=\ddb\Phi=\frac{\sqrt{-1}}{2\pi}\Phi_{i\bar{j}}dz^{i}\wedge d\bar{z}^{j}$,
$\alpha=\ddb F=\frac{\sqrt{-1}}{2\pi}F_{i\bar{j}}dz^{i}\wedge d\bar{z}^{j}$.

Assume
\[
U(q+y^{1}e_{1}^{q}+\cdots+y^{n}e_{n}^{q})=q'+\sum_{i}U^{i}(y)e_{i}^{q}
\]
Note that $\mu_{\omega}$ is the composition of $\pi$ with $X_{\Sigma}^{\geq}\rightarrow\bar{\mathcal{P}}$,
$\pi$ is independent of metrics, then compose (\ref{eq:moment map})
for $\bar{f}$ and (\ref{eq:inverse of moment map}) for $u$, we
have

\begin{equation}
U^{i}(y)=\frac{\pp\bar{f}}{\pp a_{i}}\left(\exp\frac{\pp u}{\pp y^{1}},\cdots,\exp\frac{\pp u}{\pp y^{n}}\right)\exp\frac{\pp u}{\pp y^{i}}\label{eq:transition map}
\end{equation}
Since $\exp\frac{\pp u}{\pp y^{i}}=y^{i}\exp(1+\frac{\pp v}{\pp y^{i}})$
is smooth due to the Guillemin's boundary condition, so $U$ is a
diffeomorphism.

When $y^{k}>0$ for all $k$, namely in the interior of $\bar{\mathcal{P}}$,
by $\bar{f}(a_{i})=f(\log a_{i})$,

\[
U^{i}(y)=\frac{\pp f}{\pp x_{i}}\left(\frac{\pp u}{\pp y^{1}},\cdots,\frac{\pp u}{\pp y^{n}}\right)
\]
namely $(U^{i})=\nabla f(\nabla u)$, since $\nabla f:\mathbb{R}^{n}\rightarrow\mathcal{Q}$
is a diffeomorphism and the inverse is $\nabla g$, so $\nabla u=\nabla g(U)$.
Change the order of derivatives, $u_{ij}=u_{ji}$, we see $U$ must
satisfy a compatible condition,

\begin{equation}
\sum_{k}g^{ik}(U)\frac{\pp U^{j}}{\pp y^{k}}=\sum_{k}g^{jk}(U)\frac{\pp U^{i}}{\pp y^{k}}\label{eq:compatible}
\end{equation}
by Proposition \ref{g^ij IS smooth over polytope}, $g^{ik}$ is smooth
on $\mathcal{\bar{Q}}$, so above identity is actually valid on $\mathcal{\bar{P}}$.
\begin{rem}
The compatible condition can be described in a natural way, for a
point $y\in\mathcal{P}$, we define a metric on $T_{y}M_{R}=M_{R}$
by $\langle\frac{\partial}{\partial y^{i}},\frac{\partial}{\partial y^{j}}\rangle=g_{ij}(U(y))$,
then (\ref{eq:compatible}) says $DU|_{y}:M_{R}\rightarrow M_{R}$
is self-dual and positive, positive is because $\langle DU\frac{\partial}{\partial y^{i}},\frac{\partial}{\partial y^{j}}\rangle=\frac{\pp U^{k}}{\pp y^{i}}g_{kj}(U)=f_{kl}(\nabla u)u_{li}g_{kj}(U)=u_{ji}$
is positive definite. If $y\in\partial P$, we define the metric on
the tangent space of face, and restrict $DU|_{y}$ on the tangent
space is self-dual and positive.
\end{rem}
For metric $\omega$ and $\alpha$, we defined a linear transform
$A$ on $T^{1,0}X$ by
\[
\langle A\pp_{i},\pp_{j}\rangle_{\omega}=\langle\pp_{i},\pp_{j}\rangle_{\alpha}
\]
where $\langle\pp_{i},\pp_{j}\rangle_{\omega}=\Phi_{i\bar{j}}$. Suppose
$A\pp_{i}=A_{i}^{k}\pp_{k}$, then $A_{i}^{k}=F_{i\bar{j}}\Phi^{k\bar{j}}$.
$A$ is self-dual respect to $\omega$ and $\alpha$, namely $\langle A\pp_{i},\pp_{j}\rangle_{\omega}=\langle\pp_{i},A\pp_{j}\rangle_{\omega}$,
$\langle A\pp_{i},\pp_{j}\rangle_{\alpha}=\langle\pp_{i},A\pp_{j}\rangle_{\alpha}$,
and positive $\langle Ax,x\rangle_{\omega}=\langle x,x\rangle_{\alpha}>0$
for $0\neq x\in T^{1,0}M$.

On $U_{0}\cong(\mathbb{C}^{*})^{n}$, by (\ref{eq:phi_ij on C-star})
and (\ref{eq:inverse of Big phi_ij}),
\begin{equation}
A_{i}^{k}=F_{i\bar{j}}\Phi^{k\bar{j}}=\frac{1}{z_{i}\bar{z_{j}}}f_{ij}z_{k}\bar{z}_{j}u_{kj}=\frac{z_{k}}{z_{i}}\frac{\partial U^{i}}{\partial y^{k}}\label{eq:  A  formula}
\end{equation}
The characteristic polynomial of $A:T^{1,0}M\rightarrow T^{1,0}M$
\[
\det(id+tA)=\det[\delta_{i}^{k}+tA_{i}^{k}]=\det[\frac{z_{k}}{z_{i}}(\delta_{i}^{k}+t\frac{\partial U^{i}}{\partial y^{k}})]=\det(id+tDU)|_{\mu_{\omega}}
\]
This identity holds on $X$ by continuation.
\begin{prop}
\label{prop: polynimial}$A:T^{1,0}M\rightarrow T^{1,0}M$ with $DU:M_{R}\rightarrow M_{R}$
have the same characteristic polynomial. In particular, the eigenvalues
of $DU$ are positive.
\end{prop}
Note that

\[
\det(id+tA)=\frac{(\omega+t\alpha)^{n}}{\omega^{n}}
\]
so we have
\[
\frac{(\omega+t\alpha)^{n}}{\omega^{n}}=\det(id+tDU)|_{\mu_{\omega}}
\]
In particular,
\begin{equation}
tr_{\omega}\alpha=\frac{n\omega^{n-1}\wedge\alpha}{\omega^{n}}=trA=trDU|_{\mu_{\omega}},\ \frac{\alpha^{n}}{\omega^{n}}=\det A=\det DU|_{\mu_{\omega}}\label{eq:trace formular}
\end{equation}
Donaldson's equation (\ref{eq:donaldson equation}) is $trDU\equiv nc$,
and $U$ is subject to (\ref{eq:compatible}).

Moreover, when we restrict metrics on the invariant subvariety, we
have
\begin{prop}
\label{trD(U|_F)  formula}Assume $F$ is a $p$-dimensional face
of $\mathcal{P}$, $V$ is the invariant subvariety corresponding
to $F$, $\omega_{V}$ is the restriction of $\omega$ on $V$, $U|_{F}:F\rightarrow F'$
is the restriction of $U$, $F'$ is the corresponding face of $\mathcal{Q}$,
then on $V$
\[
\frac{(\omega_{V}+t\alpha_{V})^{p}}{\omega_{V}^{p}}=\det(id+tD(U|_{F}))|_{\mu_{\omega}}
\]
where $\mu_{\omega}$ map $V$ onto $\bar{F}$, $D(U|_{F})=DU|_{TF}$
is the tangent map of $U|_{F}$, since $TF=TF'$, it is a linear transform
of $TF$.\end{prop}
\begin{proof}
Just note that $\mu_{\omega}|_{V}:V\rightarrow\bar{F}$ is the moment
map of $(V,\omega_{V})$ w.r.t. the real torus action.
\end{proof}
In particular, on $V$
\begin{equation}
\alpha_{V}\wedge\frac{\omega_{V}^{p-1}}{(p-1)!}=\left(trDU|_{TF}\right)|_{\mu_{\omega}}\frac{\omega_{V}^{p}}{p!}\label{eq:trace formula on face}
\end{equation}
Note that the push-out measure of $\frac{\omega_{V}^{p}}{p!}$ by
$\mu_{\omega}|_{V}$ is $\sigma_{F}$, the canonical measure on $F$.
First we have the Lebesgue measure $\Omega$ on $M_{R}$ induced by
the lattice $M$, and let $\{u_{i}\}_{i=1}^{n-p}$ be the generators
that vanished along $TF$, then $\sigma_{F}=(\iota_{v_{1}}\cdots\iota_{v_{n-p}}\Omega)|_{TF}$,
where $\{v_{i}\}\subset M_{R}$ such that $\langle u_{i},v_{j}\rangle=\delta_{ij}$.
Integrate (\ref{eq:trace formula on face}), we have
\[
\int_{V}\alpha_{V}\wedge\frac{\omega_{V}^{p-1}}{(p-1)!}=\int_{F}trDU|_{TF}\ d\sigma_{F}
\]
In particular, take $F=\mathcal{P}$
\[
\int_{X}\alpha\wedge\frac{\omega^{n-1}}{(n-1)!}=\int_{\mathcal{P}}\sum_{i}\frac{\partial U^{i}}{\partial y^{i}}\ dy
\]

Then the condition (\ref{eq:gabor 's condition}) for the invariant
subvarieties is for any $p$-dimensional face $F$ of $\mathcal{P}$
\begin{equation}
\frac{1}{vol(F)}\int_{F}trDU|_{TF}\ d\sigma_{F}<nc=\frac{1}{vol(\mathcal{P})}\int_{\mathcal{P}}trDU\ dy\label{eq:gabor's condition toric}
\end{equation}
If (\ref{eq:gabor 's condition}) have a solution, we can derived
this directly. The solution induces a transition map $U$ such that
$trDU\equiv nc$. Since on $F$, $TF$ is invariant under $DU$, it
induces $DU|_{M_{R}/TF}:M_{R}/TF\rightarrow M_{R}/TF$, and its trace
must be positive, so

\[
trDU=trDU|_{TF}+trDU|_{M_{R}/TF}>trDU|_{TF}
\]
integrate this inequality over $F$, we get (\ref{eq:gabor's condition toric}).

\section{the Flow of Transition Maps}

Suppose $\omega_{\varphi_{t}}$ be the solution of (\ref{eq:J-flow}),
if the initial data $\omega$, $\varphi_{0}$ are invariant then $\varphi_{t}$
is always invariant. Assume $\omega$ and $\alpha$ has potential
$\phi_{0}$ and $f$ defined on $N_{R}$ respectively, then $\omega_{\varphi_{t}}$
has potential $\phi_{t}=\phi_{0}+\varphi_{t}$, denote the Legendre
transform of $\phi_{t}$ is $u_{t}$ defined on $\bar{\mathcal{P}}$,
namely the symplectic potential of $\omega_{\varphi_{t}}$. The transition
map from $\mathcal{\bar{P}}$ to $\mathcal{\bar{Q}}$ induced by $\omega_{\varphi_{t}}$
and $\alpha$ is $U_{t}$.

Fix $y\in\mathcal{P}$, assume $x_{t}\in N_{R}$ such that
\[
d\phi_{t}(x_{t})=y
\]
then
\[
u_{t}(y)=\langle x_{t},y\rangle-\phi_{t}(x_{t})
\]
\[
\frac{\pp u_{t}}{\pp t}(y)=\langle\frac{dx_{t}}{dt},y\rangle-\frac{\pp\phi_{t}}{\pp t}(x_{t})-\langle d\phi_{t},\frac{dx_{t}}{dt}\rangle=-\frac{\pp\phi_{t}}{\pp t}(x_{t})=-\frac{\pp\wifi_{t}}{\pp t}(x_{t})
\]
by this and (\ref{eq:J-flow}), (\ref{eq:trace formular}), we have

\begin{equation}
\frac{\pp u_{t}}{\pp t}=\sum f_{ij}(\nabla u_{t})u_{ij}-nc,\ on\ \mathcal{P}\label{eq:u_t j-flow equation}
\end{equation}
and
\[
\frac{\pp u_{t}}{\pp t}=\sum_{i}\frac{\partial U^{i}}{\partial y^{i}}-nc,\ on\ \mathcal{\bar{P}}
\]
since $\frac{\pp u_{t}}{\pp t}$ is smooth on $\mathcal{\bar{P}}$.

Next we consider the evolute equation of $U_{t}$, on $\mathcal{P}$

\begin{eqnarray*}
\frac{\pp U^{i}}{\pp t} & = & \frac{\pp}{\pp t}(f_{i}(\nabla u_{t}))=\sum_{j}f_{ij}(\nabla u_{t})\frac{\pp}{\pp y^{j}}(\frac{\pp u_{t}}{\pp t})\\
 & = & \sum_{j,k}f_{ij}(\nabla u_{t})\frac{\pp^{2}U^{k}}{\pp y^{j}\pp y^{k}}=\sum_{j,k}g^{ij}(U)\frac{\pp^{2}U^{k}}{\pp y^{j}\pp y^{k}}
\end{eqnarray*}
Recall that $[g^{ij}]=[\partial_{i}\partial_{j}g]^{-1}$ is smooth
on $\bar{\mathcal{Q}}$, so actually on the whole $\mathcal{\bar{P}}$,
we have

\begin{equation}
\frac{\pp U^{i}}{\pp t}=\sum_{j,k}g^{ij}(U)\frac{\pp^{2}U^{k}}{\pp y^{j}\pp y^{k}}\label{eq:evolute U}
\end{equation}
From (\ref{eq: g^ij formular}) we know exactly how $[g^{ij}]$ degenerate
at boundary, when $U\in F'$ the vector $\left(\sum_{j,k}g^{ij}(U)\frac{\pp^{2}U^{k}}{\pp y^{j}\pp y^{k}}\right)^{i}$
is located in the tangent space of $F'$, this make $U_{t}$ map $F$
to $F'$ along the flow. In particular, $g^{jk}=0$ at the vertex,
so $U_{t}$ fix all vertex.

(\ref{eq:evolute U}) is not a parabolic system, however we can use
(\ref{eq:compatible}) to modify it, recall that

\[
\sum_{k}g^{ik}(U)\frac{\pp U^{j}}{\pp y^{k}}=\sum_{k}g^{jk}(U)\frac{\pp U^{i}}{\pp y^{k}},\ on\ \mathcal{\bar{P}}
\]
thus
\begin{eqnarray}
\frac{\pp U^{i}}{\pp t} & = & \frac{\pp}{\pp y^{k}}(g^{ij}(U)\frac{\pp U^{k}}{\pp y^{j}})-(g^{ij})_{l}(U)\frac{\pp U^{l}}{\pp y^{k}}\frac{\pp U^{k}}{\pp y^{j}}\nonumber \\
 & = & \frac{\pp}{\pp y^{k}}(g^{kj}(U)\frac{\pp U^{i}}{\pp y^{j}})-(g^{ij})_{l}(U)\frac{\pp U^{l}}{\pp y^{k}}\frac{\pp U^{k}}{\pp y^{j}}\label{eq:Evolute equation parabolic}\\
 & = & g^{kj}(U)\frac{\pp^{2}U^{i}}{\pp y^{k}\pp y^{j}}-(g^{ij})_{l}(U)\frac{\pp U^{l}}{\pp y^{k}}\frac{\pp U^{k}}{\pp y^{j}}+(g^{kj})_{l}(U)\frac{\pp U^{l}}{\pp y^{k}}\frac{\pp U^{i}}{\pp y^{j}}\nonumber
\end{eqnarray}
where $(g^{ij})_{l}=\frac{\partial}{\partial y^{l}}g^{ij}$ is the
derivatives of $g^{ij}(y)$.

(\ref{eq:Evolute equation parabolic}) is a quasi-linear parabolic
system degenerated on the boundary, the equation in second row have
a nice divergence form. A direct computation show that if the solution
of (\ref{eq:Evolute equation parabolic}) satisfies (\ref{eq:compatible})
at $t=0$ then will satisfy it all the time. We already know that
(\ref{eq:Evolute equation parabolic}) has long time solution from
the origin J-flow (\ref{eq:J-flow}), the question is how $U_{t}$
will behave as time tend to infinity.

\subsection{Some Basic Estimates}
\begin{lem}
\textup{\label{lem:-is-bounded}$\inf_{t=0}tr_{\omega_{\varphi}}\alpha\leq tr_{\omega_{\varphi}}\alpha\leq\sup_{t=0}tr_{\omega_{\varphi}}\alpha$.}\end{lem}
\begin{proof}
As in \cite{parabolic flow}, take derivative of (\ref{eq:J-flow})
w.r.t. time,
\[
\frac{\partial^{2}\varphi}{\partial t^{2}}=g_{\varphi}^{i\bar{l}}\alpha_{i\bar{j}}g_{\varphi}^{k\bar{j}}(\frac{\partial\varphi}{\partial t})_{k\bar{l}}
\]
then applied the maximal principle. We also can use (\ref{eq:evolute U}),
let $\delta>0$, suppose $G\triangleq trDU-\delta t$ at $(y_{0},t_{0})$
take the maximum value over $\mathcal{\bar{P}}\times[0,T]$. If $t_{0}>0$,
we consider the case when $y_{0}\in F^{\circ}$, $F$ is a 1-codimensional
face, the other case is similar. Choose a coordinate such that $\mathcal{P}\subset\{y^{n}\geq0\}$,
$F=\mathcal{P}\cap\{y^{n}=0\}$, then at $(y_{0},t_{0})$
\begin{eqnarray*}
0 & \leq & \frac{\partial}{\partial t}G=\sum_{p<n}\left(\frac{\partial U^{p}}{\partial t}\right)_{p}+\left(\frac{\partial U^{n}}{\partial t}\right)_{n}-\delta\\
 & = & \sum_{p,q<n}g^{pq}(U)G_{pq}+\sum_{p,q<n}\left(g^{pq}(U)\right)_{p}G_{q}\\
 &  & +\left(g^{nn}\right)_{n}(U)\frac{\partial U^{n}}{\partial y^{n}}G_{n}-\delta\\
 & \leq & -\delta
\end{eqnarray*}
Note that $[g^{pq}(U(y_{0}))]>0$, $g^{ni}(U(y_{0}))=0$, $\left(g^{nn}\right)_{n}(U(y_{0}))\geq0$,
$[G_{pq}(y_{0})]\leq0$, $G_{q}(y_{0})=0$, $G_{n}(y_{0})\leq0$,
$\frac{\partial U^{n}}{\partial y^{n}}(y_{0})>0$. It is a contradiction,
so $t_{0}=0$. Then let $\delta\rightarrow0$, we get the upper bound,
for the lower bound is similar.
\end{proof}
The above estimate gives upper bound of the eigenvalues of $DU$.
At a point, choose a basis $\{e_{i}\}$ of $T^{1,0}X$ such that $\langle e_{i},e_{j}\rangle_{\omega_{\varphi}}=\delta_{ij}$,
$\langle e_{i},e_{j}\rangle_{\alpha}=\lambda_{i}\delta_{ij}$, $\lambda_{i}>0$,
$Ae_{i}=\lambda_{i}e_{i}$, $A$ is the linear transform induced by
$\omega_{\varphi}$ and $\alpha$. From proposition \ref{prop: polynimial},
$\lambda_{i}$ is also the eigenvalue of $DU$,

\begin{equation}
\left\Vert A\right\Vert ^{2}\triangleq tr(AA^{*})=tr(A^{2})=\sum\lambda_{i}^{2}<\left(\sum\lambda_{i}\right)^{2}=(trA)^{2}\label{eq:norm of A}
\end{equation}
where $A^{*}$ is the dual transform w.r.t $\omega_{\varphi}$ or
$\alpha$. So $\left\Vert A\right\Vert ^{2}$ is bounded uniformly
along the flow, moreover on $U_{0}$, with (\ref{eq:inverse of Big phi_ij}),
(\ref{eq:  A  formula})

\[
tr(AA^{*})=A_{j}^{i}\bar{A_{l}^{k}}\Phi_{i\bar{k}}\Phi^{j\bar{l}}=A_{j}^{i}\bar{A_{l}^{k}}F_{i\bar{k}}F^{j\bar{l}}=\frac{\partial U^{j}}{\partial y^{i}}\frac{\partial U^{l}}{\partial y^{k}}g^{ik}(U)g_{jl}(U)
\]

\begin{thm}
The eigenvalues of $DU_{t}$ are positive and upper bounded uniformly.
In the interior of $\mathcal{P}$, we have a partial bound on $DU$,\textup{
\begin{equation}
\sum_{i,j,k,l}\frac{\partial U^{j}}{\partial y^{i}}\frac{\partial U^{l}}{\partial y^{k}}g^{ik}(U)g_{jl}(U)\leq C\label{eq:partial derivative estimate}
\end{equation}
and on $\bar{\mathcal{P}}$, }
\begin{equation}
\det DU\leq C'
\end{equation}

\end{thm}
To get the version of (\ref{eq:partial derivative estimate}) on $\partial\mathcal{P}$,
take a point $y$ in the interior of face $F$, $\dim F=p$. Since
$DU|_{y}:M_{R}\rightarrow M_{R}$ have an invariant subspace $TF$,
it induces $DU|_{M_{R}/TF}:M_{R}/TF\rightarrow M_{R}/TF$, because
$F$ is the intersection of $n-p$ $(p+1)$-dimensional face and $U$
map face to face, so $M_{R}/TF$ can be decomposed to a direct sum
of 1-dimensional invariant subspace, and the eigenvalues on these
1-dim subspace are also eigenvalues of $DU$, so they are positive
and bounded uniformly. For $DU|_{TF}:TF\rightarrow TF$, it is self-dual
and positive w.r.t. the metric on $T_{F}$, namely $D^{2}g|_{F'}$,
$F'$ is the corresponding face.

For example, $\dim F=n-2$, choose a coordinate such that $F$ is
parallel to $\{y^{n}=y^{n-1}=0\}$,
\begin{cor}
On the face $F$, the eigenvalues of $DU$ are constituted of $\frac{\partial U^{n}}{\partial y^{n}}$,
$\frac{\partial U^{n-1}}{\partial y^{n-1}}$ and eigenvalues of $DU|_{TF}$,
they are positive and upper bounded uniformly, and in the interior
of $F$,
\[
\sum_{i,j,k,l\leq n-2}\frac{\partial U^{j}}{\partial y^{i}}\frac{\partial U^{l}}{\partial y^{k}}g^{ik}(U)g_{jl}(U)\leq C
\]
\end{cor}
\begin{rem}
The reason that we call (\ref{eq:partial derivative estimate}) is
just a partial bound is, for $y\in\mathcal{P}^{\circ}$, choose a
basis of $M_{R}$ such that $g_{ij}(U_{t}(y))=\mu_{i}\delta_{ij}$,
then (\ref{eq:partial derivative estimate}) is $\sum_{i,j}\left(\frac{\partial U^{j}}{\partial y^{i}}\right)^{2}\frac{\mu_{j}}{\mu_{i}}\leq C$,
but when $t\rightarrow\infty$, $U_{t}(y)$ may approach to $\partial\mathcal{Q}$,
there exists some $i,j$ such that $\frac{\mu_{j}}{\mu_{i}}$ go to
zero, so we can't bound $\frac{\partial U^{j}}{\partial y^{i}}(y)$
from (\ref{eq:partial derivative estimate}). Take $i=j$, $\left|\frac{\partial U^{i}}{\partial y^{i}}\right|$
is bounded uniformly.\end{rem}
\begin{thm}
The flow converges to a smooth solution of (\ref{eq:donaldson equation})
if and only if there exists $\delta>0$ such that $\det DU\geq\delta$
uniformly.\end{thm}
\begin{proof}
The necessity is trivial. Conversely if $\det DU\geq\delta$ uniformly,
since the eigenvalues of $DU$ is upper bounded uniformly, so is below
bounded uniformly, hence $tr_{\alpha}\omega_{\varphi}$ is bounded
uniformly from both side, then by the arguments in \cite{W1}, flow
converges to the solution of (\ref{eq:donaldson equation}).
\end{proof}
As the counterpart of Calabi's functional, we have the energy functional
$E$,

\begin{equation}
E_{\alpha}(\omega)=\frac{1}{2}\int_{X}(tr_{\omega}\alpha)^{2}\frac{\omega^{n}}{n!}=\frac{1}{2}\int_{\mathcal{P}}(trDU)^{2}\ dy\label{eq:energy}
\end{equation}

\begin{prop}
Energy functional $E$ is non-increasing along the flow.\end{prop}
\begin{proof}
by (\ref{eq:evolute U})
\begin{eqnarray}
\frac{dE}{dt} & = & \int_{\mathcal{P}}trDU\frac{\partial}{\partial y^{k}}\left(\frac{\partial U^{k}}{\partial t}\right)dy\nonumber \\
 & = & \int_{\mathcal{P}}\textrm{div}\left(trDU\cdot\frac{\partial U}{\partial t}\right)dy-\int_{\mathcal{P}}\frac{\partial trDU}{\partial y^{k}}\frac{\partial U^{k}}{\partial t}\ dy\nonumber \\
 & = & -\int_{\mathcal{P}}g^{kl}(U)\frac{\partial trDU}{\partial y^{k}}\frac{\partial trDU}{\partial y^{l}}dy\leq0\label{eq:dE/dt}
\end{eqnarray}
since on the boundary of $\mathcal{P}$, $\frac{\partial U}{\partial t}$
is along the face, so the divergence term is zero.
\end{proof}
Note that we also have
\[
\frac{dE}{dt}=-\int_{\mathcal{P}}\frac{\partial trDU}{\partial y^{k}}\frac{\partial U^{k}}{\partial t}\ dy=-\int_{\mathcal{P}}g_{kl}(U)\frac{\partial U^{k}}{\partial t}\frac{\partial U^{l}}{\partial t}\ dy\leq0
\]
By \cite{parabolic flow}, we know that $E$ has the same critical
point as J-functional, namely the solution of (\ref{eq:donaldson equation}).
It is interesting to consider the variational problem

\[
\min\{\int_{\mathcal{P}}(trDU)^{2}\ dy\mid U:\bar{\mathcal{P}}\rightarrow\bar{\mathcal{Q}},\ U\ s.t.\ (\ref{eq:compatible})\}
\]

\begin{example}
\label{in-,-they  example Fang and Lai}In \cite{ansatz}, Fang and
Lai study the J-flow on the $\mathbb{P}^{n}$ blow-up 1 point under
the assumption of Calabi symmetry, namely the metrics are $U(n)$-invariant,
in our toric setting require metrics are $(\mathbb{S}^{1})^{n}$-invariant,
to make their results fit into the toric setting, we need to require
the symplectic potentials have more symmetry.
\begin{figure}
\begin{centering}
\includegraphics[bb=10bp 90bp 310bp 260bp,clip]{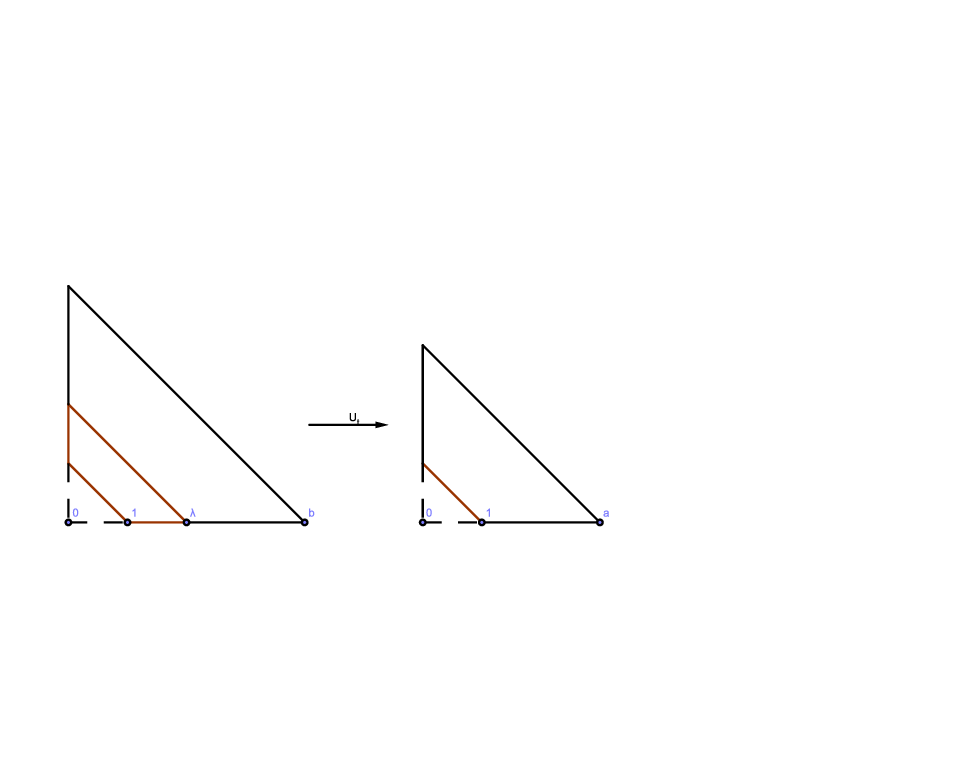}
\par\end{centering}

\caption{}
\end{figure}
Let
\[
\mathcal{P}=\left\{ y\in\mathbb{R}^{n}\mid y^{i}\geq0,\ 1\leq\sum y^{i}\leq b\right\}
\]
it gives the $\mathbb{P}^{n}$ blow-up 1 point with K\"{a}hler class $[\omega]=b[E_{\infty}]-[E_{0}]$,
$E_{\infty}$ is the pull-back of hyperplane divisor by $X\rightarrow\mathbb{P}^{n}$,
$E_{0}$ is the exceptional divisor. Let $\mathcal{Q}=\left\{ y\mid y^{i}\geq0,\ 1\leq\sum y^{i}\leq a\right\} $,
it corresponds $[\alpha]=a[E_{\infty}]-[E_{0}]$. We require the symplectic
potentials have the following form, denote $B=\sum y^{i}$
\[
u=\sum y^{i}\log y^{i}-B\log B+h(B)
\]
where $h$ is a convex function defined on $[1,b]$ and satisfies
$h(B)-(B-1)\log(B-1)-(b-B)\log(b-B)$ is smooth on $[1,b]$, then
we can check that $u$ satisfies Guillemin's conditions. In the same
way, symplectic potential defined on $\mathcal{Q}$ has form $g=\sum y^{i}\log y^{i}-B\log B+\theta(B)$,
let the Legendre transform of $h$ and $\theta$ is $p$ and $\eta$
respectively, they are defined on $\mathbb{R}$. Then the induced
transition map $U$ between polytopes is
\[
U^{i}(y)=\eta'(h'(\sum y^{i}))\frac{y^{i}}{\sum y^{i}}
\]
Let $f(B)=\eta'(h'(B))$, then $f$ is a smooth function that map
$[1,b]$ to $[1,a]$, $0<f'<\infty$.

Suppose $u(y,t)$ is a solution of J-flow (\ref{eq:u_t j-flow equation}),
it preserves its form along the flow, $h$, $p$ and $f$ changes
by time. The evolute equation (\ref{eq:evolute U}) is reduced to
\[
\frac{\partial f}{\partial t}=\frac{1}{\theta''(f)}\left(\frac{\partial^{2}f}{\partial B^{2}}+(n-1)\frac{1}{B}\frac{\partial f}{\partial B}-(n-1)\frac{1}{B^{2}}f\right)
\]
Note that $nc=n\frac{ab^{n-1}-1}{b^{n}-1}$. For the limit behavior
of flow, there are three cases which be up to $nc$.\end{example}
\begin{casenv}
\item $nc>n-1$, the flow converges to a smooth solution of (\ref{eq:donaldson equation}).
$U_{t}\rightarrow U_{\infty}$ is a diffeomorphism between polytopes
satisfies $trDU_{\infty}\equiv nc$.
\item $nc=n-1$, the flow converges to a metric with conic singularity along
$E_{0}$, and is a smooth solution of (\ref{eq:donaldson equation})
on $X\backslash E_{0}$. $U_{t}\rightarrow U_{\infty}$ is a smooth
one-to-one map between polytopes but not a diffeomorphism, $\det DU_{\infty}=0$
on face $F_{0}=\{y\mid\sum y^{i}\equiv1\}$ which corresponds $E_{0}$.
$trDU_{\infty}\equiv nc$, and on $F_{0}$, $trDU_{\infty}|_{TF_{0}}=trDU_{\infty}$.
\item $nc<n-1$, the most interesting case, $\omega_{t}\rightarrow\omega_{\infty}+(\lambda-1)[E_{0}]$
is a K\"{a}hler current, $[\omega_{\infty}]=b[E_{\infty}]-\lambda[E_{0}]$,
where $\lambda\in(1,b)$ is determinate by
\[
(n-1)\frac{b}{\lambda}+\frac{\lambda^{n-1}}{b^{n-1}}=na
\]
Note that $nc'=n\frac{[\omega_{\infty}]^{n-1}\cup[\alpha]}{[\omega_{\infty}]^{n}}=n\frac{ab^{n-1}-\lambda^{n-1}}{b^{n}-\lambda^{n}}<nc$,
the above equation is equivalent to $nc'=\frac{n-1}{\lambda}$.\\
$\omega_{\infty}$ is a metric with conic singularity along $E_{0}$,
and is a smooth solution of $c'\omega^{n}=\omega^{n}\wedge\alpha$
on $X\backslash E_{0}$.\\
$U_{t}\rightarrow U_{\infty}$ is just $\mathcal{C}^{1}$ map which
squeeze the region $\{1\leq\sum y^{i}$$\leq\lambda\}$ onto the face
$\{y\mid\sum y^{i}\equiv1\}$ of $\mathcal{Q}$, so on this region
$\det DU_{\infty}=0$, and $trDU_{\infty}=\frac{n-1}{\sum y^{i}}$.
Its second derivative jump at $\{\sum y^{i}\equiv\lambda\}$.\\
$U_{\infty}$ map $\{\lambda\leq\sum y^{i}\leq b\}$ onto $\mathcal{Q}$,
in this region satisfies $trDU_{\infty}\equiv nc'$, and $\det DU_{\infty}=0$
only on $\{\sum y^{i}\equiv\lambda\}$. In this case, we see the gradient
of symplectic potential $\nabla u_{t}$ will blow up in $\{1\leq\sum y^{i}\leq\lambda\}$.
\end{casenv}
From the partial bound of derivative (\ref{eq:partial derivative estimate}),
we can prove the following property.
\begin{prop}
Suppose $y_{t}$ is path of points in $\mathcal{P^{\circ}}$, and
there exists a domain $\Omega\subset\subset\mathcal{Q}$, such that
$U_{t}(y_{t})\in\Omega$ for $t>T$, then for any domain \textup{$\Omega_{1}$
such that} $\Omega\subset\subset\Omega_{1}\subset\subset\mathcal{Q}$,
there exits $\epsilon>0$, such that $B_{\epsilon}(y_{t})\subset\mathcal{P}$,
and \textup{$U_{t}(B_{\epsilon}(y_{t}))\subset\Omega_{1}$ for $t>T$.
$B_{\epsilon}(y_{t})$ is the Euclidean ball, the distance and length
in the proof is w.r.t. the Euclidean metric.}\end{prop}
\begin{proof}
assume $d(\bar{\Omega},\partial\Omega_{1})>\delta$, there exits a
point $b\in\partial U_{t}^{-1}(\Omega_{1})$ such that $d(y_{t},b)=d(y_{t},\partial U_{t}^{-1}(\Omega_{1}))$,
let $l$ be the segment located in $U_{t}^{-1}(\bar{\Omega}_{1})$
connected $y_{t}$ and $b$, since $U_{t}(U_{t}^{-1}(\bar{\Omega}_{1}))=\bar{\Omega}_{1}\subset\subset\mathcal{P}$,
by (\ref{eq:partial derivative estimate}) we know on $U_{t}^{-1}(\bar{\Omega}_{1})$
the derivatives $\sum\left(\frac{\partial U^{j}}{\partial y^{i}}\right)^{2}\leq C$
for all time, so the length of curve $\mathbf{L}(U_{t}(l))\leq C'd(y_{t},b)$,
and $U_{t}(l)$ connect $U_{t}(y_{t})\in\Omega$ and $U_{t}(b)\in\partial\Omega_{1}$,
so $\mathbf{L}(U_{t}(l))\geq\delta$. Thus $d(y_{t},\partial U_{t}^{-1}(\Omega_{1}))=d(y_{t},b)\geq\delta/C'$,
then $B_{\delta/2C'}(y_{t})\subset U_{t}^{-1}(\Omega_{1})\subset\mathcal{P}$
for $t>T$. \end{proof}
\begin{rem}
By this property, we take a point $z\in\mathcal{Q^{\circ}}$, the
inverse image $U_{t}^{-1}(z)=y_{t}$, then $d(y_{t},\partial\mathcal{P})>\epsilon$.
In particularly, the distance from the minimum point of $u_{t}$ to
$\partial\mathcal{P}$ has a uniform lower bound.
\end{rem}
Finally, we make some speculation. First, if we can prove that $y_{t}\rightarrow y_{\infty}\in\mathcal{P^{\circ}}$,
then $U_{t}(y_{\infty})\rightarrow z$, by the above proposition,
there exists $\delta>0$ such that $B_{\delta}(y_{\infty})\subset\mathcal{P^{\circ}}$
and $d(U_{t}(B_{\delta}(y_{\infty}),\partial\mathcal{Q})>c$, then
on $B_{\delta}(y_{\infty})$, we have uniform derivative bound by
(\ref{eq:partial derivative estimate}) and (\ref{eq:Evolute equation parabolic})
is strictly parabolic, we may show $U_{t}$ converges on this ball.
Then union these balls together we get a open set $\Theta\subset\mathcal{P}$,
$U_{t}\rightarrow U_{\infty}$ on $\Theta$, and $U_{\infty}(\Theta)=\mathcal{Q}$
for the arbitrariness of $z$. This means that $U_{t}$ finally squeeze
$\mathcal{P}\backslash\Theta$ onto $\partial\mathcal{Q}$, as the
case 3 in the example.

$\mathbf{Acknowledgements.}$ I would like to thank my advisor Gang
Tian for constant encouragement, Yalong Shi for many times useful
discussion and Jiaqiang Mei for his help.

\address{\textsc{Department of Mathematics and Institute of Mathematical Science,
Nanjing University, Nanjing, 210093, Jiangsu province, China} \\
E-mail address : \textsl{yeeyoe@163.com}}

\begin{thebibliography}{10}
\bibitem[1]{Li anmin}Chen, B. H., Li, A. M., Sheng, L.: Extremal
metrics on toric surfaces. arXiv:1008.2607, 2010.

\bibitem[2]{lower bound}Chen, X. X.: On the lower bound of the Mabuchi
energy and its application. Int. Math. Res. Notices 12 (2000), 607\textendash{}623.

\bibitem[3]{parabolic flow}\textemdash{}\textemdash{}: A new parabolic
flow in K\"{a}hler manifolds. Comm. Anal. Geom. 12, no.4 (2004), 837\textendash{}852.

\bibitem[4]{moment map}Donaldson, S. K.: Moment maps and diffeomorphisms.
Asian J. Math. 3, no. 1 (1999), 1\textendash{}16.

\bibitem[5]{Scalar curvature and stability}\textemdash{}\textemdash{}:
Scalar curvature and stability of toric varieties. Journal of Differential
Geometry, 2002, 62(2): 289-349.

\bibitem[6]{toric surfaces}\textemdash{}\textemdash{}: Constant scalar
curvature metrics on toric surfaces. Geometric and Functional Analysis,
2009, 19(1): 83-136.

\bibitem[7]{ansatz}Fang, H., Lai, M. J.: Convergence of general inverse
$\sigma_{k}$-flow on K\"{a}hler manifolds with Calabi ansatz. Transactions
of the American Mathematical Society, 2013, 365(12): 6543-6567.

\bibitem[8]{FLSW}Fang, H., Lai, M. J., Song, J., Weinkove, B.: The
J-flow on K\"{a}hler surfaces: a boundary case. Analysis \& PDE, 2014,
7(1): 215-226.

\bibitem[9]{fulton}Fulton, W.: Introduction to toric varieties, Princeton
University Press, 1993.

\bibitem[10]{STABILITY}Lejmi, M., Sz\'{e}kelyhidi, G.: The J-flow and
stability. arXiv:1309.2821, 2013.

\bibitem[11]{ross an obsruction}Ross, J., Thomas, R.: An obstruction
to the existence of constant scalar curvature K\"{a}hler metrics. Journal
of Differential Geometry, 2006, 72(3): 429-466.

\bibitem[12]{SW}Song, J., Weinkove, B.: The convergence and singularities
of the J-flow with applications to the Mabuchi energy. Comm. Pure
Appl. Math. 61 (2008), no. 2, 210\textendash{}229.

\bibitem[13]{newest}\textemdash{}\textemdash{}: The Degenerate J-flow
and the Mabuchi energy on minimal surfaces of general type. Universitatis
Iagellonicae Acta Mathematica, 2014, 50.

\bibitem[14]{filtration}Sz\'{e}kelyhidi, G.: Filtrations and test-configurations.
arXiv:1111.4986, 2011.

\bibitem[15]{tian book}Tian, G.: Canonical metrics in K\"{a}hler geometry,
Springer, 2000.

\bibitem[16]{ZHU WANG}Wang, X. J., Zhu, X. H.: K\"{a}hler\textendash{}Ricci
solitons on toric manifolds with positive first Chern class. Advances
in Mathematics, 2004, 188(1): 87-103.

\bibitem[17]{W1}Weinkove, B.: Convergence of the J -flow on K\"{a}hler
surfaces, Comm. Anal. Geom. 12, no. 4 (2004), 949\textendash{}965.

\bibitem[18]{W2}\textemdash{}\textemdash{}: On the J-flow in higher
dimensions and the lower boundedness of the Mabuchi energy, J. Differential
Geom. 73 (2006), no. 2, 351\textendash{}358.

\end{thebibliography}
\end{document}